\documentclass[a4paper,11pt,twoside]{article}
\usepackage{latexsym,amsfonts,index,amssymb}
\usepackage[centertags]{amsmath}
\usepackage{amstext}
\usepackage{enumerate}
\usepackage{cases}
\usepackage{color}
\usepackage{enumitem}
\usepackage{bbold}
\usepackage{graphicx}
\usepackage{stmaryrd}
\usepackage{calrsfs}

\usepackage[arrow,curve,rotate]{xypic}

\DeclareSymbolFont{rsfscript}{OMS}{rsfs}{m}{n}
\DeclareSymbolFontAlphabet{\mathrsfs}{rsfscript}
\renewcommand{\mathcal}{\mathrsfs}

\def\titlerunning#1{\gdef\titrun{#1}}
\makeatletter
\def\author#1{\gdef\autrun{\def\and{\unskip, }#1}\gdef\@author{#1}}
\def\address#1{{\def\and{\\\hspace*{18pt}}\renewcommand{\thefootnote}{}%
\footnote {#1}}%
\markboth{\autrun}{\titrun}} \makeatother
\def\email#1{e-mail: #1}

\def\keywords#1{\par\medskip
\noindent\textbf{Keywords.} #1}

\setenumerate[1]{label=$(\alph*)$,leftmargin=*}
\setenumerate[2]{label=$(\roman*)$,leftmargin=*}

\DeclareMathAlphabet{\mathpzc}{OT1}{pzc}{m}{it}
\definecolor{verde}{rgb}{0,.5,0}

\DeclareSymbolFont{bbold}{U}{bbold}{m}{n}
\DeclareSymbolFontAlphabet{\mathbbold}{bbold}

\numberwithin{equation}{section}

\def\N{{\mathbb N}}

\newfont{\sss}{cmssi10 at 11pt}
\newfont{\bss}{cmssbx10 at 11pt}
\newfont{\tit}{cmitt10 at 11pt}
\newcommand{\Se}{{\bf S}}

\newcommand{\LG}{{\bf LG}}

\newcommand{\Sl}{{\bf Sl}}
\newcommand{\LSl}{{\bf LSl}}
\newcommand{\LV}{{\bf LV}}
\newcommand{\K}{{\bf K}}
\newcommand{\D}{{\bf D}}
\newcommand{\G}{{\bf G}}

\newcommand{\J}{{\bf J}}

\newcommand{\V}{{\bf V}}

\newcommand{\NA}{A^{-\N}}

\newcommand{\om}[2] {{\overline{\Omega}}_{#1}{\mathbf #2}}

\newcommand{\ome}[2] {{{\Omega}}_{#1}{\mathbf #2}}
\newcommand{\omek}[3] {{{\Omega}}^{#1}_{#2}{\mathbf #3}}

\newcommand{\mfrg}[3] {{#1}:{#2}\mathop{\hbox{\kern5pt$\circ$\kern-12pt\raise0.1pt\hbox
{$\longrightarrow$}}}{#3}}

\newcommand{\infe}[1]{{#1}^{\!{\scriptscriptstyle{-\!\infty}}}}

\newtheorem{theorem}{Theorem}[section]
\newtheorem{proposition}[theorem]{Proposition}
\newtheorem{corollary}[theorem]{Corollary}

\newtheorem{lemma}[theorem]{Lemma}

\newtheorem{definition}[theorem]{Definition}

\newenvironment{definition*}{\begin{trivlist}\item[\hskip
    \labelsep{\bf Definition\quad}]}%
  {\hfill\qed\end{trivlist}}
\newenvironment{notation*}{\begin{trivlist}\item[\hskip
    \labelsep{\bf Notation\quad}]}%
  {\end{trivlist}}

  \def\qed{{\unskip\nobreak\hfil\penalty50\hskip .001pt\hbox{}%
      \nobreak\hfil
      \vrule height 1.2ex width 1.1ex depth -.1ex
      \parfillskip=0pt\finalhyphendemerits=0\medbreak}}
\newenvironment{proof}{\begin{trivlist}\item[\hskip%
     \labelsep{\bf Proof.\quad}]}%
 {\hfill\qed\rm\end{trivlist}}

  {\hfill\qed\end{trivlist}}%

\begin{document}

\titlerunning{On $\kappa$-reducibility of  pseudovarieties of the form  $\bf V*\bf D$}

\title{\bf On $\kappa$-reducibility of  pseudovarieties\\  of the form  $\bf V*\bf D$}

\author{J. C. Costa %
  \and %
  C. Nogueira %
  \and %
  M. L. Teixeira%
}

\date{February 8, 2016}

\maketitle

\address{ %
  J. C. Costa \& M. L. Teixeira: %
  CMAT, Dep.\ Matem\'{a}tica e Aplica\c{c}\~{o}es, Universidade do Minho, Campus
  de Gualtar, 4710-057 Braga, Portugal; %
  \email{jcosta@math.uminho.pt, mlurdes@math.uminho.pt} %
  \and %
  C. Nogueira: %
   CMAT, Escola Superior de Tecnologia e Gest\~ao,
  Instituto Polit\'ecnico de Leiria,
  Campus 2, Morro do Lena, Alto Vieiro, 2411-901
   Leiria, Portugal; %
   \email{conceicao.veloso@ipleiria.pt} %
}

\begin{abstract}
This paper deals with the reducibility property of semidirect products of the form $\bf V*\bf D$ relatively to
graph equation systems, where \D\ denotes the pseudovariety of definite semigroups. We show that, if the
pseudovariety $\bf V$ is  reducible with respect to the canonical signature $\kappa$ consisting of the
multiplication and the $(\omega-1)$-power, then $\bf V*\bf D$ is also reducible with respect to $\kappa$.

\end{abstract}
\section{Introduction}
A semigroup (resp.\ monoid) \emph{pseudovariety} is a class of finite semigroups (resp.\ monoids) closed under taking subsemigroups (resp.\ submonoids), homomorphic images and finite direct products. It is said to be \emph{decidable} if there is an algorithm to test membership of a finite semigroup (resp.\ monoid) in that pseudovariety.  The semidirect product of pseudovariets has been getting much attention, mainly due to the Krohn-Rhodes decomposition theorem~\cite{Krohn&Rhodes:1965}. In turn,
the pseudovarieties of the form $\V * \D $, where \D\ is the pseudovariety of all finite semigroups whose
idempotents are right zeros, are among the most studied semidirect
products~\cite{Straubing:1985,Tilson:1987,Almeida&Azevedo:1993,Almeida:1995,Almeida&Azevedo&Teixeira:2000}. For
a pseudovariety \V\ of monoids, \LV\ denotes the pseudovariety of all finite semigroups $S$ such that $eSe\in\V$
for all idempotents $e$ of $S$. We know
from~\cite{EilenbergB:1976,Straubing:1985,Therien&Weiss:1985,Tilson:1987} that $\V *\D$ is contained in \LV\ and
that $\V *\D=\LV$ if and only if \V\ is \emph{local} in the sense of Tilson~\cite{Tilson:1987}. In particular,
the equalities $\Sl *\D=\LSl$ and $\G *\D=\LG$ hold for the pseudovarieties $\Sl$ of semilattices and $\G$  of
groups.

It is known that the semidirect product operator does not preserve decidability of pseudovarieties
\cite{Rhodes:1999,Auinger&Steinberg:2003}. The notion of \emph{tameness} was introduced by Almeida and
Steinberg~\cite{Almeida&Steinberg:2000,Almeida&Steinberg:2000b}  as a tool for proving decidability of
semidirect products. The fundamental property for tameness is \emph{reducibility}. This property was originally
formulated in terms of graph equation systems and latter extended to any system of
equations~\cite{Almeida:2002,Rhodes&Steinberg:2009}. It is parameterized by an \emph{implicit signature}
$\sigma$ (a set of implicit operations on semigroups containing the multiplication), and
 we speak of $\sigma$-reducibility.  For short, given an equation system $\Sigma$ with rational constraints,
 a pseudovariety $\V$ is $\sigma$-reducible relatively to $\Sigma$
  when the existence of a solution of $\Sigma$ by implicit operations over ${\bf V}$
 implies the existence of a solution  of $\Sigma$ by $\sigma$-words over ${\bf V}$ and satisfying the same constraints.  The
pseudovariety \V\ is said to be \emph{$\sigma$-reducible} if it is $\sigma$-reducible with respect to every
finite graph equation system. The implicit signature which is most commonly encountered in the literature is the
\emph{canonical signature} $\kappa=\{ab,a^{\omega-1}\}$ consisting of the multiplication and the
$(\omega-1)$-power. For instance,  the pseudovarieties ${\bf D}$~\cite{Almeida&Zeitoun:2003}, ${\bf
G}$~\cite{Ash:1991,Almeida&Steinberg:2000b}, ${\bf J}$~\cite{Almeida:1995,Almeida:2002} of all finite ${\cal
J}$-trivial semigroups, ${\bf LSl}$~\cite{Costa&Teixeira:2004} and ${\bf R}$~\cite{Almeida&Costa&Zeitoun:2005}
of all finite ${\cal R}$-trivial semigroups are $\kappa$-reducible.

In this paper, we study the $\kappa$-reducibility property of semidirect products of the form $\V * \D$. This research is essentially inspired by the papers~\cite{Costa&Nogueira&Teixeira:2015b,Costa&Teixeira:2004} (see also~\cite{Costa&Nogueira:2009} where a stronger form of $\kappa$-reducibility was established for $\LSl$). We prove that,  if \V\  is $\kappa$-reducible then $\V*\D$ is  $\kappa$-reducible. In particular, this gives a new
and simpler  proof (though with the same basic idea) of the $\kappa$-reducibility of $\LSl$ and establishes the
$\kappa$-reducibility of the pseudovarieties $\LG$, $\J*\D$ and ${\bf R}*\D$. Combined with the recent proof that the
 $\kappa$-word problem for ${\bf LG}$ is decidable~\cite{Costa&Nogueira&Teixeira:2015}, this shows that $\LG$ is $\kappa$-tame, a problem proposed by Almeida a few years ago. This also extends part of our work
in the paper~\cite{Costa&Nogueira&Teixeira:2015b}, where we proved that under mild hypotheses on an implicit
signature $\sigma$,  if \V\  is  $\sigma$-reducible relatively to \emph{pointlike} systems of equations (i.e.,
systems of equations of the form $x_1=\cdots=x_n$) then $\V*\D$ is pointlike $\sigma$-reducible as well. As
in~\cite{Costa&Nogueira&Teixeira:2015b}, we use results from~\cite{Almeida&Costa&Teixeira:2010}, where various
kinds of $\sigma$-reducibility of semidirect products with an order-computable pseudovariety were considered.
More specifically, we know from~\cite{Almeida&Costa&Teixeira:2010} that a pseudovariety of the form $\V*\D_k$ is
$\kappa$-reducible when \V\  is $\kappa$-reducible, where $\D_k$ is the order-computable pseudovariety defined
by the identity $yx_1\cdots x_k=x_1\cdots x_k$. As $\V*\D=\bigcup_{k}\V*\D_k$, we utilize this result as a way
to achieve our property concerning the pseudovarieties $\V*\D$. The method used in this paper is similar to that of~\cite{Costa&Nogueira&Teixeira:2015b}. However, some significant changes, inspired by~\cite{Costa&Teixeira:2004}, had to be introduced in order to deal with the much more intricate graph equation systems.
%%%%%%%%%%%%%%%%%%%%%%%%%%%%%%%%%%%%%%%%%%%%%%%%%%%%%%%%%%%%%%%%%%%%%%%%%%%%%%%%%%%%%%%%%%%%%%%%%%%%%%%%%
\section{Preliminaries}\label{section:preliminaries}
The reader is referred to the standard bibliography on finite semigroups,
namely~\cite{Almeida:1995,Rhodes&Steinberg:2009}, for general background and undefined terminology. For basic
definitions and results about combinatorics on words, the reader may wish to consult~\cite{Lothaire:2002}.

%%%%%%%%%%%%%%%%%%%%%%%%%%%%%%%%%%%%%%%%%%%%%%%%%%%%%%%%%%%%%%%%%%%%%%%%%%%%%%%%%%%%%%%%%%%%%%%%%%%%%%%%%
\subsection{Words and pseudowords}\label{subsection:pseudowords}
Throughout this paper, $A$ denotes a  finite non-empty set called an \emph{alphabet}. The \emph{free semigroup}
and the \emph{free monoid} generated by $A$ are denoted respectively by $A^+$ and $A^*$. The empty word is
represented by $1$ and the length of a word $w \in A^*$ is denoted by $|w|$. A word is called {\em primitive} if
it cannot be written in the form $u^n$ with $n>1$.  Two words $u$ and $v$ are said to be {\em conjugate} if
 $u=w_1w_2$ and $v=w_2w_1$ for some words $w_1,w_2\in A^*$. A {\em Lyndon word} is a primitive word
which is minimal in its conjugacy class, for the lexicographic order on $A^+$.

A  {\em left-infinite} word on $A$ is a sequence $w=(a_n)_{n}$ of letters of $A$ indexed by  $-\N$ also written
$w=\cdots a_{-2}a_{-1}$. The set of all left-infinite words on $A$ will be denoted  by $\NA$ and we put
$A^{-\infty}=A^+\cup\NA$. The set $A^{-\infty}$ is endowed with a semigroup structure by defining a product as
follows: if $w,z\in A^+$, then $wz$ is already defined; left-infinite words are right zeros; finally, if
$w=\cdots a_{-2}a_{-1}$ is a left-infinite word and $z=b_1b_2\cdots b_n$ is a finite word, then $wz$ is the
left-infinite word $wz=\cdots a_{-2}a_{-1}b_1b_2\cdots b_n$. A left-infinite word $w$ of the form
$\infe{u}v=\cdots uuuv$, with $u\in A^+$ and $v\in A^*$, is said to be {\em ultimately periodic}. In case $v=1$,
the word $w$ is named {\em periodic}. For a periodic word $w=\infe{u}$, if $u$ is a primitive word, then it will
be called the {\em root} of $w$ and its length $|u|$ will be said to be the \emph{period} of $w$.

For a pseudovariety \V\ of semigroups, we denote by  $\om{A}{\V}$ the relatively free
pro-\V\  semigroup   generated by the set $A$: for each pro-\V\ semigroup $S$ and each function $\varphi:A\to
S$, there is a unique continuous homomorphism $\overline{\varphi}:\om {A}{\V}\to S$ extending $\varphi$.  The
elements of  $\om{A}{\V}$ are called \emph{pseudowords} (or \emph{implicit operations}) over \V. A pseudovariety \V\ is called {\em order-computable} when the subsemigroup  $\ome{A}{\V}$ of $\om{A}{\V}$ generated by $A$ is finite, in which case $\ome{A}{\V}=\om{A}{\V}$, and effectively computable. Recall that, for the pseudovariety
$\Se$ of all finite semigroups, $\ome{A}{\Se}$  is (identified with) the free semigroup $A^+$. The elements of  $\om{A}{\Se}\setminus A^+$ will then be called \emph{infinite pseudowords}.

A \emph{pseudoidentity} is a formal equality $\pi=\rho$ of pseudowords $\pi,\rho\in\om{A}{\Se}$ over $\Se$.  We say that $\V$ \emph{satisfies} the pseudoidentity $\pi=\rho$, and write $\V\models \pi=\rho$, if $\varphi\pi=\varphi\rho$ for every continuous homomorphism $\varphi:\om{A}{\Se}\to S$ into a semigroup $S\in\V$, which is equivalent to saying that  $p_{\V} \pi=p_{\V}\rho$ for the \emph{natural
projection} $p_\V:\om{A}{\Se}\rightarrow \om{A}{\V}$.

\subsection{Pseudoidentities over $\V*\D_k$}\label{subsection:pseudoidentities_V*Dk}
For a positive integer $k$, let $\D_k$ be the pseudovariety of all finite semigroups satisfying the identity
$yx_1\cdots x_k=x_1\cdots x_k$. Denote by $A^k$ the set of words over $A$ with length $k$ and by $A_k$ the set
$\{w\in A^+: |w|\leq k\}$  of non-empty words over $A$ with length at most $k$. We notice that $\ome{A}{\D_k}$
may be identified with the semigroup whose support set is $A_k$ and whose
 multiplication is given by $u\cdot v={\mathtt t}_k(uv)$, where  ${\mathtt t}_k w$ denotes the longest
 suffix of length at most $k$ of a given (finite or left-infinite) word $w$. Then, the $\D_k$ are order-computable pseudovarieties such that $\D=\bigcup_{k}\D_k$.  Moreover, it is well-known that $\om{A}{\D}$  is isomorphic to the semigroup  $A^{-\infty}$.

 For each pseudoword $\pi\in\om{A}{\Se}$, we denote by  ${\mathtt t}_k\pi$ the unique smallest word (of $A_{k}$) such that $\D_k\models\pi={\mathtt t}_k\pi$. Simetrically, we denote by  ${\mathtt i}_k\pi$ the smallest word (of $A_{k}$) such that $\K_k\models\pi={\mathtt i}_k\pi$,  where $\K_k$ is the dual pseudovariety of $\D_k$ defined by the identity
$x_1\cdots x_ky=x_1\cdots x_k$. Let $\Phi_k$ be the function $A^+\rightarrow (A^{k+1})^*$  that sends each word
$w\in A^+$ to the sequence of factors of length $k+1$ of $w$, in the order they occur in $w$. We still denote by
$\Phi_k$ (see~\cite{Almeida&Azevedo:1993} and~\cite[Lemma 10.6.11]{Almeida:1995}) its unique continuous
extension $\om{A}{\Se}\rightarrow (\om{A^{k+1}}{\Se})^1$. This function $\Phi_k$ is a $k$-\emph{superposition
homomorphism}, with the meaning that it verifies the conditions:
\begin{enumerate}\vspace{-0.2cm}
  \item[i)] $\Phi_k w=1$ for every $w\in A_{k}$;\vspace{-0.2cm}
  \item[ii)]  $\Phi_k(\pi\rho)=\Phi_k \pi \Phi_k\bigl(({\mathtt t}_k\pi)\rho\bigr)= \Phi_k\bigl(\pi({\mathtt i}_k\rho)\bigr) \Phi_k \rho$ for every $\pi,\rho\in \om{A}{\Se}$.\vspace{-0.1cm}
\end{enumerate}

Throughout the paper, $\V$ denotes a non-locally trivial pseudovariety of semigroups. For any pseudowords
$\pi,\rho\in \om{A}{\Se}$, it is known from~\cite[Theorem~10.6.12]{Almeida:1995} that
\begin{equation}\label{eq:char-psid-VDk}
\V*\D_k\models \pi=\rho\quad\Longleftrightarrow\quad {\mathtt i}_k\pi={\mathtt i}_k\rho,\ {\mathtt
t}_k\pi={\mathtt t}_k\rho\mbox{ and }\V\models \Phi_k\pi=\Phi_k\rho.
\end{equation}

%%%%%%%%%%%%%%%%%%%%%%%%%%%%%%%%%%%%%%%%%%%%%%%%%%%%%%%%%%%%%%%%%%%%%%%%%%%%%%%%%%%%%%%%%%%%%%%%%%%%%%%%%
\subsection{Implicit signatures and $\sigma$-reducibility}\label{subsection:signatures}
 By an \emph{implicit signature}  we mean a set $\sigma$ of pseudowords (over $\Se$) containing the
 multiplication. In particular, we represent by $\kappa$ the implicit signature $\{ab,a^{\omega-1}\}$, usually called the
  \emph{canonical signature}. Every profinite semigroup has a natural structure of a $\sigma$-algebra, via the natural interpretation of
pseudowords on profinite semigroups.  The $\sigma$-subalgebra of $\om{A}{\Se}$ generated by $A$ is denoted
 by $\omek{\sigma}{A}{\Se}$. It is freely generated by $A$ in the variety of $\sigma$-algebras generated by the pseudovariety $\Se$
 and its elements are called $\sigma$-\textit{words} (over \Se).  To a  (directed multi)graph $\Gamma=V(\Gamma)\uplus E(\Gamma)$, with vertex set $V(\Gamma)$,
  edge set $E(\Gamma)$, and edges $\alpha {\sf e}\xrightarrow {\sf e}\omega {\sf e}$, we associate the system $\Sigma_\Gamma$ of
all equations of the form $(\alpha {\sf e})\, {\sf e}= \omega {\sf e}$, with ${\sf e}\in E(\Gamma)$.  Let $S$ be a finite $A$-generated semigroup, $\delta:\om{ A}{\Se}\to S$ be the continuous homomorphism respecting the choice of generators and $\varphi:\Gamma\to S^1$ be an evaluation mapping such that $ \varphi E(\Gamma)\subseteq S$. We say that a mapping $\eta:\Gamma\to (\om{ A}{\Se})^1$ is a {\em
\V-solution} of $\Sigma_\Gamma$ with respect to $(\varphi,\delta)$ when $\delta\eta=\varphi$ and $\V\models
\overline{\eta} u =\overline{\eta} v$ for all $(u=v)\in\Sigma_\Gamma$. Furthermore, if  $\eta \Gamma\subseteq(\omek{\sigma}{A}{S})^1$ for an implicit signature~$\sigma$, then  $\eta$ is called a $(\V,\sigma)$-\emph{solution}. The pseudovariety \V\ is said to be {\em $\sigma$-reducible} relatively to the system $\Sigma_\Gamma$ if the existence of a \V-solution of $\Sigma_\Gamma$ with respect to a pair $(\varphi,\delta)$ entails the existence of a $(\V
,\sigma)$-solution of $\Sigma_\Gamma$ with respect to the same pair $(\varphi,\delta)$.  We say that $\V$ is
\emph{$\sigma$-reducible}, if it is $\sigma$-reducible relatively to $\Sigma_\Gamma$ for all finite graphs $\Gamma$.
%%%%%%%%%%%%%%%%%%%%%%%%%%%%%%%%%%%%%%%%%%%%%%%%%%%%%%%%%%%%%%%%%%%%%%%%%%%%%%%%%%%%%%%%%%%%%%%%%%%%%%%%%%%%%%%%%%%%%%%%%%%%%%%%
%%%%%%%%%%%%%%%%%%%%%%%%%%%%%%%%%%%%%%%%%%%%%%%%%%%%%%%%%%%%%%%%%%%%%%%%%%%%%%%%%%%%%%%%%%%%%%%%%%%%%%%%%%%%%%%%%%%%%%%%%%%%%%%%
\section{$\kappa$-reducibility of $\V*\D$}
Let $\V$ be a given $\kappa$-reducible non-locally trivial pseudovariety. The purpose of this paper is to prove
the $\kappa$-reducibility of the pseudovariety $\V*\D$. So, we fix a finite graph $\Gamma$ and a finite
$A$-generated semigroup $S$ and consider a $\V*\D$-solution $\eta:\Gamma\to (\om{A}{\Se})^1$  of the system
$\Sigma_\Gamma$ with respect to a pair $(\varphi,\delta)$, where $\varphi:\Gamma\to S^1$ is an evaluation
mapping such that $\varphi E(\Gamma) \subseteq S$ and $\delta:\om{ A}{\Se}\to S$ is a continuous homomorphism
respecting the choice of generators. We have to construct a $(\V*\D,\kappa)$-solution $\eta':\Gamma\to
(\omek{\kappa}{A}{S})^1$ of $\Sigma_\Gamma$ with respect to the same pair $(\varphi,\delta)$.
%%%%%%%%%%%%%%%%%%%%%%%%%%%%%%%%%%%%%%%%%%%%%%%%%%%%%%%%%%%%%%%%%%%%%%%%%%%%%%%%%%%%%%%%%%%%%%%%%%%%%%%%%%%%%%%%%%%%%%%%%%%%%%%%
%%%%%%%%%%%%%%%%%%%%%%%%%%%%%%%%%%%%%%%%%%%%%%%%%%%%%%%%%%%%%%%%%%%%%%%%%%%%%%%%%%%%%%%%%%%%%%%%%%%%%%%%%%%%%%%%%%%%%%%%%%%%%%%%
\subsection{Initial considerations}\label{section:initial_considerations}
Suppose that ${\sf g}\in\Gamma$ is such that $\eta{\sf g}=u$ with $u\in A^*$. Since $\eta$ and $\eta'$ are
supposed to be $\V*\D$-solutions of the system $\Sigma_\Gamma$ with respect to $(\varphi,\delta)$, we must have
$\delta\eta=\varphi=\delta\eta'$ and so, in particular,
 $\delta\eta'{\sf g}=\delta u$. As the homomorphism  $\delta:\om{A}{\Se}\to {S}$ is arbitrarily fixed, it may happen that the equality $\delta\eta'{\sf g}=\delta u$ holds only when $\eta'{\sf g}=u$.  In that case we would be obliged to define $\eta'{\sf g}=u$.
 Since we want to describe an algorithm to define $\eta'$ that should work for any given graph and solution, we will then construct a solution $\eta'$
 verifying the following condition:
 \begin{equation}
 \forall {\sf g}\in\Gamma,\quad (\eta{\sf g}\in A^*\Longrightarrow \eta'{\sf g}=\eta{\sf g}).\tag*{$\mathcal C_1(\Gamma,\eta,\eta')$}
 \end{equation}

Suppose next that a vertex ${\sf v}\in V(\Gamma)$ is such that $\D\models\eta{\sf v}=u^\omega$ with $u\in A^+$, that is, suppose that $p_\D\eta{\sf v}=\infe{u}$. Because $\Gamma$ is an arbitrary graph, it could include, for instance, an edge ${\sf e}$ such that $\alpha {\sf
e}=\omega{\sf e}={\sf v}$ and the labeling $\eta$ could be such that $\eta{\sf e}=u$. Since $\D$ is a
subpseudovariety of $\V*\D$, $\eta$ is a $\D$-solution of $\Sigma_\Gamma$ with respect to
$(\varphi,\delta)$. Hence, as by condition $\mathcal C_1(\Gamma,\eta,\eta')$ we want to preserve finite labels, it would
follow in that case that $\D\models(\eta'{\sf v})u=\eta'{\sf v}$ and, thus, that $\D\models\eta'{\sf
v}=u^\omega=\eta{\sf v}$. This observation suggests that we should preserve the projection into $\om{A}{\D}$ of labelings of
vertices ${\sf v}$ such that  $p_\D\eta{\sf v}=\infe{u}$ with $u\in A^+$. More generally, we will construct the
$(\V*\D,\kappa)$-solution $\eta'$ in such a way that the following condition holds:
\begin{equation}
 \forall {\sf v}\in V(\Gamma),\quad (p_\D\eta{\sf v}=\infe{u} z\mbox{ with $u\in A^+$ and $z\in A^*$}\Longrightarrow p_\D\eta'{\sf v}=p_\D\eta{\sf v}).\tag*{$\mathcal C_2(\Gamma,\eta,\eta')$}
 \end{equation}

Let $\ell_\eta=\mbox{max}\{|u|: \mbox{$u\in A^*$ and $\eta{\sf g}=u$ for some ${\sf g}\in\Gamma$}\}$ be the
maximum length of finite labels under $\eta$ of elements of $\Gamma$. To be able to make some reductions on the
graph $\Gamma$ and solution $\eta$, described in Section~\ref{section:simplifications}, we want $\eta'$ to
verify the extra condition below, where $L\geq\ell_\eta$ is a non-negative integer to be specified later, on
Section~\ref{section:borders}:
\begin{equation}
 \forall {\sf v}\in V(\Gamma),\quad (\eta{\sf v}=u\pi\mbox{ with $u\in A_L$}\Longrightarrow \eta'{\sf v}=u\pi'\mbox{ with $\delta \pi=\delta \pi'$}).\tag*{$\mathcal C_3(\Gamma,\eta,\eta')$}
 \end{equation}
%%%%%%%%%%%%%%%%%%%%%%%%%%%%%%%%%%%%%%%%%%%%%%%%%%%%%%%%%%%%%%%%%%%%%%%%%%%%%%%%%%%%%%%%%%%%%%%%%%%%%%%%%%%%%%%%%%%%%%%%%%%%%%%
%%%%%%%%%%%%%%%%%%%%%%%%%%%%%%%%%%%%%%%%%%%%%%%%%%%%%%%%%%%%%%%%%%%%%%%%%%%%%%%%%%%%%%%%%%%%%%%%%%%%%%%%%%%%%%%%%%%%%%%%%%%%%%%%
\subsection{Simplifications on the solution $\eta$}\label{section:simplifications}
We begin this section by reducing to the case in which all vertices of $\Gamma$ are labeled by infinite
pseudowords under $\eta$. Suppose first that there is an edge ${\sf v}\xrightarrow {\sf e}{\sf w}$ such that
$\eta{\sf v}=u_{\sf v}$ and $\eta{\sf e}=u_{\sf e}$ with $u_{\sf v}\in A^*$ and $u_{\sf e}\in A^+$, so that
$\eta{\sf w}=u_{\sf v}u_{\sf e}$. Drop the edge ${\sf e}$ and consider the restrictions $\eta_1$ and
$\varphi_1$, of $\eta$ and $\varphi$ respectively, to the graph $\Gamma_{\!1}=\Gamma\setminus\{{\sf e}\}$. Then
$\eta_1$ is a $\V*\D$-solution of the system $\Sigma_{\Gamma_{\!1}}$ with respect to the pair
$(\varphi_1,\delta)$. Assume that there is a $(\V*\D,\kappa)$-solution $\eta'_1$ of $\Sigma_{\Gamma_{\!1}}$ with
respect to $(\varphi_1,\delta)$ verifying condition $\mathcal C_1(\Gamma_{\!1},\eta_1,\eta'_1)$. Then
$\eta'_1{\sf v}=u_{\sf v}$ and $\eta'_1{\sf w}=u_{\sf v}u_{\sf e}$. Let $\eta'$ be the extension of $\eta'_1$ to
$\Gamma$ obtained by letting $\eta'{\sf e}=u_{\sf e}$. Then $\eta'$ is a $(\V*\D,\kappa)$-solution  of
$\Sigma_{\Gamma}$ with respect to $(\varphi,\delta)$. By induction on the number of edges labeled by finite
words under $\eta$ beginning in vertices also labeled by finite words under $\eta$, we may therefore assume that
there are no such edges in $\Gamma$.

Now, we remove all vertices ${\sf v}$ of $\Gamma$  labeled by finite words under $\eta$ such that ${\sf v}$ is not the beginning of an edge, thus obtaining a graph $\Gamma_1$. As above, if $\eta'_1$ is a $(\V*\D,\kappa)$-solution of $\Sigma_{\Gamma_{\!1}}$, then we build a $(\V*\D,\kappa)$-solution $\eta'$ of $\Sigma_{\Gamma}$ by letting $\eta'$ coincide with
$\eta'_1$ on $\Gamma_1$ and letting $\eta'{\sf v}=\eta{\sf v}$  for each vertex ${\sf v}\in\Gamma\setminus \Gamma_1$. So, we may assume that all vertices of $\Gamma$  labeled by finite words under $\eta$  are the beginning of some edge.

Suppose next that ${\sf v}\xrightarrow {\sf e}{\sf w}$ is an edge such that $\eta{\sf v}=u$ and $\eta{\sf
e}=\pi$ with $u\in A^*$ and $\pi\in \om{ A}{\Se}\setminus A^+$. Notice that, since it is an infinite pseudoword,
$\pi$ can be written as $\pi=\pi_1\pi_2$ with both $\pi_1$ and $\pi_2$ being infinite pseudowords.  Drop the
edge ${\sf e}$ (and the vertex ${\sf v}$ in case ${\sf e}$ is the only edge beginning in ${\sf v}$) and let
${\sf v}_1$ be a new vertex and  ${\sf v}_1\xrightarrow {{\sf e}_1}{\sf w}$  be a new edge thus obtaining a new
graph $\Gamma_{\!1}$. Let $\eta_1$ and $\varphi_1$ be the labelings of $\Gamma_{\!1}$ defined as follows:
\begin{itemize}

\item  $\eta_1$ and $\varphi_1$  coincide,
respectively, with $\eta$ and $\varphi$ on $\Gamma'=\Gamma_{\!1}\cap \Gamma$;

\item    $\eta_1{\sf v}_1=u\pi_1$, $\eta_1{\sf e}_1=\pi_2$, $\varphi_1{\sf v}_1=\delta\eta_1{\sf v}_1$ and $\varphi_1{\sf
e}_1=\delta\eta_1{\sf e}_1$.
\end{itemize}
Then $\eta_1$ is  a $\V*\D$-solution of the system $\Sigma_{\Gamma_{\!1}}$ with respect to the pair
$(\varphi_1,\delta)$. Assume that there is a $(\V*\D,\kappa)$-solution $\eta'_1$ of $\Sigma_{\Gamma_{\!1}}$ with
respect to $(\varphi_1,\delta)$ verifying conditions $\mathcal C_1(\Gamma_{\!1},\eta_1,\eta'_1)$ and $\mathcal
C_3(\Gamma_{\!1},\eta_1,\eta'_1)$. In particular, since $L$ is chosen to be greater than $\ell_\eta$,
$\eta'_1{\sf v}_1=u\pi'_1$ with $\delta \pi_1=\delta \pi'_1$. Let $\eta'$ be the extension of
${\eta'_1}_{|\Gamma'}$ to $\Gamma$ obtained by letting $\eta'{\sf e}=\pi'_1(\eta'_1{\sf e}_1)$ (and $\eta'{\sf
v}=u$ in case ${\sf v}\not\in\Gamma'$). As one can easily verify, $\eta'$ is a $(\V*\D,\kappa)$-solution  of
$\Sigma_{\Gamma}$ with respect to $(\varphi,\delta)$. By induction on the number of edges beginning in vertices
labeled by finite words under $\eta$, we may therefore assume that all vertices of $\Gamma$ are labeled by
infinite pseudowords under $\eta$.

Suppose at last that an edge  ${\sf e}\in\Gamma$ is labeled under $\eta$ by a finite word $u=a_1\cdots a_n$, where
$n>1$ and $a_i\in A$. Denote ${\sf v}_0=\alpha{\sf e}$ and ${\sf v}_n=\omega{\sf e}$. In this case, we drop the
edge ${\sf e}$ and, for each $i\in\{1,\ldots,n-1\}$, we add a new vertex ${\sf v}_i$ and a new edge  ${\sf v}_{i-1}\xrightarrow{{\sf
e}_i}{\sf v}_i$ to the graph $\Gamma$. Let $\Gamma_{\!1}$ be the graph thus obtained and let $\eta_1$ and $\varphi_1$ be the labelings of $\Gamma_{\!1}$ defined as follows:
\begin{itemize}

\item  $\eta_1$ and $\varphi_1$  coincide, respectively, with $\eta$ and $\varphi$ on $\Gamma'=\Gamma\setminus \{{\sf e}\}$;

\item  for each $i\in\{1,\ldots,n-1\}$, $\eta_1{\sf v}_i=(\eta{\sf v})a_1\cdots a_i$, $\eta_1{\sf e}_i=a_i$, $\varphi_1{\sf v}_i=\delta\eta_1{\sf v}_i$
 and $\varphi_1{\sf e}_i=\delta\eta_1{\sf e}_i$.
\end{itemize}
Hence, $\eta_1$ is  a $\V*\D$-solution of the system $\Sigma_{\Gamma_{\!1}}$ with respect to the pair
$(\varphi_1,\delta)$. Suppose there exists a $(\V*\D,\kappa)$-solution $\eta'_1$ of $\Sigma_{\Gamma_{\!1}}$ with respect to
$(\varphi_1,\delta)$ verifying condition  $\mathcal C_1(\Gamma_{\!1},\eta_1,\eta'_1)$. Let $\eta'$ be the extension of
${\eta'_1}_{|\Gamma'}$ to $\Gamma$ obtained by letting $\eta'{\sf e}=u$. Then $\eta'$ is a $(\V*\D,\kappa)$-solution  of
$\Sigma_{\Gamma}$ with respect to $(\varphi,\delta)$. By induction on the number of edges labeled by finite words under $\eta$, we may further assume that each edge of $\Gamma$ labeled by a finite word under $\eta$ is, in fact, labeled by a letter of the alphabet.

%%%%%%%%%%%%%%%%%%%%%%%%%%%%%%%%%%%%%%%%%%%%%%%%%%%%%%%%%%%%%%%%%%%%%%%%%%%%%%%%%%%%%%%%%%%%%%%%%%%%%%%%%%%%%%%%%%%%%%%%%%%%%%%
%%%%%%%%%%%%%%%%%%%%%%%%%%%%%%%%%%%%%%%%%%%%%%%%%%%%%%%%%%%%%%%%%%%%%%%%%%%%%%%%%%%%%%%%%%%%%%%%%%%%%%%%%%%%%%%%%%%%%%%%%%%%%%%%
\subsection{Borders of the solution $\eta$}\label{section:borders}
The main objective of this section is to define a certain class of finite words, called \emph{borders of the
solution $\eta$}. Since the equations (of $\Sigma_\Gamma$) we have to deal with are of the form $(\alpha {\sf
e})\, {\sf e}= \omega {\sf e}$, these borders will serve to signalize the transition from a vertex $\alpha {\sf
e}$ to the edge ${\sf e}$.

 For each vertex ${\sf v}$ of $\Gamma$, denote by ${\bf d}_{\sf
v}\in\NA$ the projection $p_\D\eta {\sf v}$ of $\eta {\sf v}$ into $\om A{\D}$ and  let
 $D_\eta=\{{\bf d}_{\sf v}\mid {\sf v}\in V(\Gamma)\}$. We say that two left-infinite words $v_1,v_2\in\NA$ are {\em confinal} if they have a common prefix $y\in\NA$,
that is, if $v_1=yz_1$ and $v_2=yz_2$ for some words $z_1,z_2\in A^*$. As one easily verifies, the relation
$\propto$ defined, for each ${\bf d}_{{\sf v}_1}, {\bf d}_{{\sf v}_2}\in D_\eta$, by
$${\bf d}_{{\sf v}_1}\propto {\bf d}_{{\sf v}_2}\quad\mbox{if and only if\quad ${\bf d}_{{\sf v}_1}$ and ${\bf d}_{{\sf v}_2}$ are confinal}$$
is an equivalence on $D_\eta$. For each $\propto$-class $\Delta$, we fix  a word $y_\Delta\in\NA$ and words
$z_{\sf v}\in A^*$, for each vertex ${\sf v}$ with ${\bf d}_{\sf v}\in \Delta$, such that
$${\bf d}_{\sf v}=y_\Delta z_{\sf v}.$$
Moreover, when ${\bf d}_{\sf v}$ is ultimately periodic, we choose $y_\Delta$ of the form $\infe{u}$, with $u$ a
Lyndon word, and fix $z_{\sf v}$ not having $u$ as a prefix. The word $u$ and its length $|u|$ will  be said to be, respectively, a \emph{root} and a
\emph{period} of the solution $\eta$. Without loss of generality, we assume that $\eta$ has at least one root
(otherwise we could, easily, modify the graph and the solution in order to include one).

We fix a few of the integers that will be used in the construction of the $(\V*\D,\kappa)$-solution $\eta'$.
They depend only on the mapping $\eta$ and on the semigroup $S$.
\begin{definition}[constants $n_S$, $p_\eta$, $L$, $E$   and $Q$] \label{def:constants} We let:
 \begin{enumerate}[label=$\ \;\bullet$,labelsep=*,leftmargin=4ex,parsep=0ex]
\item $n_{\;\!\!S}$ be the \emph{exponent} of $S$ which, as one recalls, is the least integer
such that $s^{n_S}$ is idempotent for every element $s$ of the finite $A$-generated semigroup $S$;

\item $p_\eta={\rm lcm}\{|u|: \mbox{ $u\in A^+$ is a root of $\eta$}\}$;

\item $L={\rm max}\{\ell_\eta,|z_{\sf v}|: {\sf v}\in V(\Gamma)\}$;

\item $E$ be an integer such that $E\geq n_{\;\!\!S}p_\eta$ and, for each word $w\in A^E$, there is a factor $e\in A^+$ of $w$ for which
$\delta e$ is an idempotent of $S$. Notice that, for each root $u$ of $\eta$, $|u^{n_{\;\!\!S}}|\leq E$ and $\delta
(u^{n_{\;\!\!S}})$ is an idempotent of $S$;

\item $Q=L+E$.
\end{enumerate}
\end{definition}

 For each positive integer $m$, we denote by $B_{m}$ the set
$$B_{m}=\{{\mathtt t}_my_\Delta\in A^m\mid \mbox{$\Delta$ is a $\propto$-class}\}.$$
If $y_\Delta=\infe{u}$ is a periodic left-infinite word, then the element $y={\mathtt t}_my_\Delta$ of $B_{m}$
will be said to be \emph{periodic} (with \emph{root} $u$ and \emph{period} $|u|$).  For words $y_1,y_2\in B_m$,
we define the \emph{gap between $y_1$ and $y_2$} as the positive integer
$$g(y_1,y_2)=\mbox{min}\{|u|\in\N: \mbox{$u\in A^+$ and, for some $v\in A^+$, $y_1u=vy_2$ or $y_2u=vy_1$}\},$$
and notice that $g(y_1,y_2)=g(y_2,y_1)\leq m$.

\begin{proposition}\label{prop:borders}
Consider the constant $Q$ introduced in Definition~\ref{def:constants}.  There exists $q_Q\in\N$ such that for all
integers $m\geq q_Q$ the following conditions hold:
\begin{enumerate}
  \item If $y_1$ and $y_2$ are distinct elements of $B_m$, then   $g(y_1,y_2)>Q$;
  \item If  $y$ is a non-periodic element of $B_m$, then  $g(y,y)>Q$.
\end{enumerate}
\end{proposition}
\begin{proof} Suppose that, for every $q_Q\in\N$ there is an integer $m\geq q_Q$ and elements $y_{m,1}$ and $y_{m,2}$
of $B_m$ such that $g(y_{m,1},y_{m,2})\leq Q$. Hence, there exist a strictly increasing sequence $(m_i)_i$ of
positive integers and an integer $r\in\{1,\ldots,Q\}$ such that $\bigl(g(y_{m_i,1},y_{m_i,2})\bigr)_i$ is
constant and equal to $r$. Moreover, since the graph $\Gamma$ is finite, we may assume that $y_{m_i,1}={\mathtt
t}_{m_i}y_{\Delta_1}$ and $y_{m_i,2}={\mathtt t}_{m_i}y_{\Delta_2}$ for every $i$ and some $\propto$-classes
$\Delta_1$ and $\Delta_2$. It then follows that $y_{\Delta_1}u=y_{\Delta_2}$ or $y_{\Delta_2}u=y_{\Delta_1}$ for
some word $u\in A^r$. Hence, $y_{\Delta_1}$ and $y_{\Delta_2}$ are confinal left-infinite words, whence
${\Delta_1}$ and ${\Delta_2}$ are the same $\propto$-class ${\Delta}$. Therefore, for every $m$, $y_{m,1}$ and
$y_{m,2}$ have the same length and are suffixes of the word $y_{\Delta}$ and, so, $y_{m,1}$ and $y_{m,2}$ are
the same word. This proves already $(a)$. Now, notice that $y_{\Delta}u=y_{\Delta}$, meaning that $y_{\Delta}$
is the  periodic left-infinite word $u^{-\infty}$. This shows $(b)$ and completes the proof of the proposition.
\end{proof}

We now fix two more integers.
\begin{definition}[constants $M$   and $k$] \label{def:constants2} We let:
 \begin{enumerate}[label=$\ \;\bullet$,labelsep=*,leftmargin=4ex,parsep=0ex]
\item $M$ be an integer such that $M$ is a multiple of $p_\eta$ and $M$ is greater than or equal to the integer $q_Q$ of Proposition~\ref{prop:borders},  and notice that $M>Q$;

\item $k=M+Q$.
\end{enumerate}
\end{definition}
 The
elements of the set $B_{M}$ will be called the \emph{borders of the solution $\eta$}. We remark that the borders
of $\eta$ are finite words of length $M$ such that, by Proposition~\ref{prop:borders}, for any two distinct
occurrences of borders $y_1$ and $y_2$ in a finite word, either these occurrences  have a gap of size at least
$Q$ between them, or $y_1$ and $y_2$ are the same periodic border $y$. In this case, $y$ is a power of its root
$u$, since $M$ is a multiple of the period $|u|$, and $g(y,y)$ is $|u|$.
%%%%%%%%%%%%%%%%%%%%%%%%%%%%%%%%%%%%%%%%%%%%%%%%%%%%%%%%%%%%%%%%%%%%%%%%%%%%%%%%%%%%%%%%%%%%%%%%%%%%%%%%%%%%%%%%%%%%%%%%%%%%%%%
%%%%%%%%%%%%%%%%%%%%%%%%%%%%%%%%%%%%%%%%%%%%%%%%%%%%%%%%%%%%%%%%%%%%%%%%%%%%%%%%%%%%%%%%%%%%%%%%%%%%%%%%%%%%%%%%%%%%%%%%%%%%%%%%
\subsection{Getting a $(\V*\D_k,\kappa)$-solution}\label{section:VDk-solution}
As $\V*\D_k$ is a subpseudovariety of $\V*\D$, $\eta$ is a $\V*\D_k$-solution of $\Sigma_\Gamma$ with respect to $(\varphi,\delta)$. The given
pseudovariety $\V$ was assumed to be $\kappa$-reducible. So, by~\cite[Corollary
6.5]{Almeida&Costa&Teixeira:2010}, $\V*\D_k$ is $\kappa$-reducible too. Therefore, there is a
$(\V*\D_k,\kappa)$-solution $\eta'_k:\Gamma\to (\omek{\kappa}{A}{S})^1$ of $\Sigma_\Gamma$ with respect to the
same pair $(\varphi,\delta)$. Moreover, as observed in~\cite[Remark 3.4]{Almeida&Costa&Zeitoun:2005}, one can
constrain the values  $\eta'_k{\sf g}$ of each ${\sf g}\in\Gamma$ with respect to properties which can be tested
in a finite semigroup. Since the prefixes and the suffixes of length at most $k$ can be tested in the finite
semigroup $\ome{A}{\K_k}\times\ome{A}{\D_k}$, we may assume further that $\eta'_k{\sf g}$ and $\eta{\sf g}$ have
the same prefixes and the same suffixes of length at most $k$. We then denote
$${\tt i}_{\sf g}={\tt i}_k\eta'_k{\sf g}={\tt i}_k\eta{\sf g}\quad\mbox{and}\quad {\tt t}_{\sf g}={\tt t}_k\eta'_k{\sf g}={\tt t}_k\eta{\sf g},$$
for each ${\sf g}\in\Gamma$. Notice that, by the simplifications introduced in
Section~\ref{section:simplifications}, if
 $\eta{\sf g}$ is a finite word, then ${\sf g}$ is an edge and $\eta{\sf g}$ is a letter $a_{\sf g}$ and so
${\tt i}_{\sf g}={\tt t}_{\sf g}=a_{\sf g}$. Otherwise, ${\tt i}_{\sf g}$ and ${\tt t}_{\sf g}$ are length $k$
words. In particular, condition $\mathcal C_1(\Gamma,\eta,\eta'_k)$ holds. That is, $\eta'_k{\sf e}=\eta{\sf e}$
for every edge ${\sf e}$ such that $\eta{\sf e}$ is a finite word. On the other hand, Lemma 2.3 (ii)
of~\cite{Costa:2004}, which is stated only for edges, can be extended easily to vertices, so that $\eta'_k{\sf
g}$ can be assumed  to be an infinite pseudoword for every ${\sf g}\in \Gamma$ such that $\eta{\sf g}$ is
infinite. Thus, in particular, $\eta'_k{\sf v}$ is an infinite pseudoword for all vertices ${\sf v}$.

Notice that, for each vertex ${\sf v}$, there exists a border $y_{\sf v}$ of $\eta$ such that
the finite word $y_{\sf v} z_{\sf v}$ is a suffix of $\eta{\sf v}$. On the other hand, by Definitions~\ref{def:constants} and~\ref{def:constants2}, $|z_{\sf v}|\leq L<Q$  and $k=M+Q$. So, as $|y_{\sf v}|=M$,
\begin{equation}\label{eq:eta'k_v}
{\tt t}_{\sf v}=x_{\sf v}y_{\sf v} z_{\sf v}\quad\mbox{and}\quad\eta'_k{\sf v}=\pi_{\sf v}{\tt t}_{\sf v}
 \end{equation}
for some infinite $\kappa$-word $\pi_{\sf v}$ and some word $x_{\sf v}\in A^+$ with $|x_{\sf v}|= Q-|z_{\sf
v}|$.
%%%%%%%%%%%%%%%%%%%%%%%%%%%%%%%%%%%%%%%%%%%%%%%%%%%%%%%%%%%%%%%%%%%%%%%%%%%%%%%%%%%%%%%%%%%%%%%%%%%%%%%%%%%%%%%%%%%%%%%%%%%%%%%%
%%%%%%%%%%%%%%%%%%%%%%%%%%%%%%%%%%%%%%%%%%%%%%%%%%%%%%%%%%%%%%%%%%%%%%%%%%%%%%%%%%%%%%%%%%%%%%%%%%%%%%%%%%%%%%%%%%%%%%%%%%%%%%%%
\subsection{Basic transformations}
The objective of this section is to introduce the basic steps that will allow to transform the $(\V*\D_k,\kappa)$-solution $\eta'_k$ into a
$(\V*\D,\kappa)$-solution $\eta'$. The process of construction of $\eta'$ from $\eta'_k$ is close to the one
used in~\cite{Costa&Nogueira&Teixeira:2015b} to handle with systems of pointlike equations. Both procedures are
supported by (basic) transformations of the form $$a_1\cdots a_{k}\mapsto a_1\cdots a_{j} (a_{i}\cdots
a_j)^{\omega}a_{j+1}\cdots  a_k,$$ which replace words of length $k$ by $\kappa$-words.
 Those procedures differ  in the way the indices $i\leq j$ are determined. In the pointlike case, the only condition
 that a basic transformation had to comply with was that $j$ had to be minimum such that the value of the word $a_1\cdots a_{k}$ under $\delta$ is
 preserved. In the present case, the basic transformations have to preserve the value under $\delta$ as well, but the equations
 $(\alpha{\sf e}){\sf e}=\omega{\sf e}$ impose an extra restriction that is not required by pointlike equations.
 Indeed, we need  $\eta'$ to verify, in particular, $\delta\eta'\alpha{\sf e}=\delta\eta'_k \alpha{\sf e}(=\delta\eta \alpha{\sf
 e})$ and $\delta\eta'{\sf e}=\delta\eta'_k {\sf e}(=\delta\eta {\sf e})$. So, somewhat informally, for a word $a_1\cdots a_{k}$ that has
an occurrence overlapping both the factors $\eta'_k  \alpha{\sf e}$ and $\eta'_k  {\sf e}$ of the pseudoword
$(\eta'_k  \alpha{\sf e})(\eta'_k  {\sf e})$, the introduction of the factor $(a_{i}\cdots  a_j)^{\omega}$ by
the basic transformation should be done either in $\eta'_k  \alpha{\sf e}$ or in $\eta'_k  {\sf e}$, and not in
both simultaneously. The borders of the solution $\eta$ were introduced to help us to deal with this extra
 restriction. Informally speaking, the borders will be used to detect the ``passage'' from the labeling under $\eta'_k$ of
 a vertex $\alpha{\sf e}$ to the labeling of the edge ${\sf e}$ and to avoid that the introduction of $(a_{i}\cdots  a_j)^{\omega}$
  affect the labelings under $\delta$ of $\eta'_k\alpha{\sf e}$ or $\eta'_k{\sf e}$.

Consider an arbitrary word $w=a_1\cdots a_n\in A^+$. An integer $m\in\{M,\ldots,n\}$ will be called a
\emph{bound of $w$} if the factor $w_{[m]}=a_{m'}\cdots a_m$ of $w$ is a border, where $m'=m-M+1$. The bound $m$ will be said to be \emph{periodic} or \emph{non-periodic} according to the border $w_{[m]}$ is periodic or not. If $w$ admits
bounds, then there is a maximum one that we name the \emph{last bound of $w$}.  In this case, if $\ell$ is
the last bound of $w$, then the border $w_{[\ell]}$ will be called the \emph{last border} of $w$. Notice
that, by Proposition~\ref{prop:borders} and the choice of $M$, if $m_1$ and $m_2$ are two bounds of $w$ with
$m_1<m_2$, then either $m_2-m_1>Q$ or $w_{[m_1]}$ and $w_{[m_2]}$ are  the same periodic border.

Let $w=a_1\cdots a_k\in A^+$ be a word of length $k$. Notice that, since $k=M+Q$, if $w$ has a
non-periodic last bound $\ell$, then $\ell$ is the unique bound of $w$. We split the word $w$ in two parts, ${\tt
l}_w$ (the \emph{left-hand of $w$}) and ${\tt r}_{\!w}$ (the \emph{right-hand of $w$}), by setting $${\tt l}_w=a_1\cdots
a_{s}\quad\mbox{and}\quad {\tt r}_w=a_{s+1}\cdots a_{k}$$ where $s$ (the \emph{splitting point of $w$}) is defined as follows:
if $w$ has a last bound $\ell$ then $s=\ell$; otherwise $s=k$. In case $w$ has a
periodic last bound $\ell$, the splitting point $s$ will be said to be  \emph{periodic}. Then, $s$ is not periodic in two situations: either $w$ has a non-periodic last border or $w$ has not a last border.  The factorization
 $$w={\tt l}_w{\tt r}_w$$
  will be called the \emph{splitting factorization of $w$}. We have $s\geq M>Q\geq E$. So, by definition of $E$, there exist integers $i$ and $j$ such that
$s-E<i<j\leq s$ and the factor $e=a_i\cdots a_{j}$ of ${\tt l}_w$ verifies $\delta e=(\delta e)^2$. We begin by fixing the maximum such $j$ and, for that $j$, we fix next an integer $i$ and a word ${\tt e}_w=a_i\cdots a_{j}$, called the \emph{essential factor of $w$}, as follows. Notice that, if the splitting point $s$ is periodic and $u$ is the root of the last border of $w$, then $\delta(u^{n_{\!S}})$ is idempotent and the  left-hand of $w$ is of the form ${\tt l}_w={\tt l}'_wu^{n_S}$. Hence, in this case, $j=s$ and we let ${\tt e}_w=u^{n_{\!S}}$, thus defining $i$ as $j-n_{\!S}|u|+1$. Suppose now that the splitting point is not periodic. In this case we let
 $i$ be the maximum integer such that $\delta(a_i\cdots a_{j})$ is idempotent. The word $w$ can be factorized as $w={\tt l}'_w{\tt e}_w{\tt l}''_w{\tt r}_w,$ where ${\tt l}'_w=a_1\cdots a_{i-1}$. We then denote by
$\widehat w$ the following $\kappa$-word
$$\widehat w={\tt l}'_w{\tt e}_w{\tt e}_w^\omega{\tt l}''_w{\tt r}_w=a_1\cdots a_{j} (a_{i}\cdots  a_j)^{\omega}a_{j+1}\cdots  a_k$$
and notice that $\delta \widehat w=\delta w$. Moreover $|{\tt e}_w {\tt l}''_w|\leq E$ and so $|{\tt l}'_w|\geq
M-E>Q-E=L$. It is also convenient to introduce two
  $\kappa$-words derived from $\widehat w$
\begin{equation}\label{eq:lambda_varrho}
\lambda_k w=a_1\cdots a_{j} (a_{i}\cdots  a_j)^{\omega},\quad \varrho_k w= (a_{i}\cdots  a_j)^{\omega}a_{j+1}\cdots  a_k.
\end{equation}
 This defines two mappings $\lambda_k,\varrho_k: A^k\rightarrow\omek{\kappa}{A}{\Se}$ that can be extended to $\om{A}{\Se}$  as done  in~\cite{Costa&Nogueira&Teixeira:2015b}. Although they are not formally the same mappings used in that paper, because of the different choice of the integers $i$ and $j$, we keep the same notation since the selection process of those integers is absolutely irrelevant for the purpose of the mappings. That is, with the above adjustment the mappings maintain the properties stated in~\cite{Costa&Nogueira&Teixeira:2015b}.

The next lemma presents a property of the $\widehat{\ \ }$-operation that is fundamental to our purposes.
\begin{lemma}\label{lemma:correct_order}
For a word $w=a_1\cdots a_{k+1}\in A^+$ of length $k+1$, let $w_1=a_1\cdots a_{k}$ and $w_2=a_2\cdots a_{k+1}$ be the two factors of $w$ of length $k$. If $\widehat w_1=a_1\cdots a_{j_1} (a_{i_1}\cdots  a_{j_1})^{\omega}a_{j_1+1}\cdots  a_k$ and  $\widehat w_2=a_2\cdots a_{j_2} (a_{i_2}\cdots  a_{j_2})^{\omega}a_{j_2+1}\cdots  a_{k+1}$, then $a_1{\tt l}_{w_2}={\tt l}_{w_1}x$ for some word $x\in A^*$. In particular $j_1\leq j_2$.
\end{lemma}
\begin{proof} Write $w_2=b_1\cdots b_k$ with $b_i=a_{i+1}$. Let $s_1$ and $s_2$ be the splitting points of $w_1$ and $w_2$ respectively, whence ${\tt l}_{w_1}=a_1\cdots a_{s_1}$ and ${\tt l}_{w_2}=b_1\cdots b_{s_2}=a_2\cdots a_{s_2+1}$. To prove that there exists a word $x$ such that $a_1{\tt l}_{w_2}={\tt l}_{w_1}x$, we have to show that $s_1\leq s_2+1$. Under this hypothesis, we then deduce that $a_{i_1}\cdots  a_{j_1}$ is an occurrence of the essential factor ${\tt e}_{w_1}$  in  ${\tt l}_{w_2}$ which proves that $j_1\leq j_2$.

 Assume first that $w_1$ has a last bound $\ell_1$, in which case $s_1=\ell_1$. By definition, $\ell_1\geq M$. If $\ell_1>M$, then the last border
of $w_1$ occurs in $w_2$, one position to the left relatively to $w_1$. Hence $\ell_1-1$ is a bound of $w_2$ and, so, $w_2$ has a last bound $\ell_2$ such that $\ell_2\geq \ell_1-1$. It follows in this case that  $s_2=\ell_2$ and  $s_1\leq s_2+1$. Suppose now that $\ell_1=M$. Since $s_2\geq M$ by definition, the condition $s_1\leq s_2+1$ holds trivially in this case. Suppose now that $w_1$ has not a last bound. Then $s_1=k$. Moreover, either $w_2$ does not have a last
bound or $k$ is its last bound. In both circumstances $s_2=k$, whence $s_1=s_2\leq s_2+1$. This concludes the proof of the lemma.
\end{proof}

In the conditions of the above lemma and as in~\cite{Costa&Nogueira&Teixeira:2015b},   we define  $\psi_k:(\om{A^{k+1}}{\Se})^1\rightarrow (\om{A}{\Se})^1$ as the only continuous monoid homomorphism which extends the mapping
\begin{alignat*}{2}
      A^{k+1} &    \rightarrow \omek{\kappa}{A}{\Se}
    \\    a_1\cdots a_{k+1} & \mapsto
   (a_{i_1}\cdots  a_{j_1})^{\omega} a_{j_1+1}\cdots a_{j_2}(a_{i_2}\cdots  a_{j_2})^{\omega}
   \end{alignat*}
and let $\theta_k=\psi_k  \Phi_k$. The function $\theta_k:\om{A}{\Se}\to (\om{A}{\Se})^1$  is a continuous
$k$-superposition homomorphism since it is the composition of the continuous $k$-superposition homomorphism
$\Phi_k$ with the continuous homomorphism $\psi_k$. We remark that a word $w=a_1\cdots a_n$ of length $n>k$ has
precisely $r=n-k+1$ factors of length $k$ and
\begin{alignat*}{2}
\theta_k (w) & =  \psi_k (a_1\cdots a_{k+1},a_2\cdots a_{k+2},\ldots , a_{r-1}\cdots a_n)\\
     & =  \psi_k (a_1\cdots a_{k+1}) \psi_k (a_2\cdots a_{k+2})\cdots \psi_k (a_{r-1}\cdots a_n)\\
    & =  (e_1^{\omega} f_1 e_2^{\omega})(e_2^{\omega} f_2e_3^{\omega})\cdots   (e_{r-1}^{\omega} f_{r-1} e_r^{\omega})\\
    & =  e_1^{\omega} f_1 e_2^{\omega} f_2\cdots   e_{r-1}^{\omega} f_{r-1} e_r^{\omega}
           \end{alignat*}
 where, for each $p\in\{1,\ldots , r\}$, $e_p$ is the essential factor ${\tt e}_{w_p}=a_{i_p}\cdots  a_{j_p}$
 of the word $w_p=a_p \cdots a_{k+p-1}$ and $f_p=a_{j_p+1} \cdots    a_{j_{p+1}}$ ($p\neq r$). Above, for each $p\in\{2,\ldots , r-1\}$, we have replaced each expression $e_p^{\omega}e_p^{\omega}$  with  $e_p^{\omega}$ since, indeed, these expressions represent the same $\kappa$-word. More generally, one can certainly replace an expression of the form $x^\omega x^n x^\omega$ with $x^\omega x^n$. Using this reduction rule as long as possible,
  $\theta_k (w)$ can be written as
  $$\theta_k (w) = e_{n_1}^{\omega} {\bar f}_1 {e}_{n_2}^\omega {\bar f}_2\cdots   {e}_{n_q}^{\omega} {\bar f}_{q},$$
 called the \emph{reduced form of $\theta_k (w)$}, where $q\in\{1,\ldots, r\}$, $1={n_1}< n_2<\cdots <n_q\leq r$,  ${\bar f}_p=f_{n_p}\cdots  f_{{n_{p+1}}-1}$ (for $p\in\{1,\ldots,q-1\}$) and ${\bar f}_q$ is $f_{n_q}\cdots  f_{{r}-1}$ if $n_q\neq r$ and it is the empty word otherwise.

%%%%%%%%%%%%%%%%%%%%%%%%%%%%%%%%%%%%%%%%%%%%%%%%%%%%%%%%%%%%%%%%%%%%%%%%%%%%%%%%%%%%%%%%%%%%%%%%%%%%%%%%%%%%%%%%%%%%%%%%%%%%%%%%
%%%%%%%%%%%%%%%%%%%%%%%%%%%%%%%%%%%%%%%%%%%%%%%%%%%%%%%%%%%%%%%%%%%%%%%%%%%%%%%%%%%%%%%%%%%%%%%%%%%%%%%%%%%%%%%%%%%%%%%%%%%%%%%%
\subsection{Definition of the $(\V*\D,\kappa)$-solution $\eta'$}
We are now in conditions to describe the procedure to transform the $(\V*\D_k,\kappa)$-solution $\eta'_k$ into the
$(\V*\D,\kappa)$-solution $\eta'$. The mapping $\eta':\Gamma\to(\omek{\kappa}{A}{S})^1$ is defined, for each ${\sf g}\in\Gamma$, as
 $$\eta'{\sf g}=(\tau_1{\sf g})  (\tau_2{\sf g})  (\tau_3{\sf g}),$$
where, for each $i\in\{1,2,3\}$,  $\tau_i :\Gamma\to(\omek{\kappa}{A}{S})^1$ is a function defined as follows.

First of all, we let $$\tau_2=\theta_k\eta'_k.$$ That $\tau_2$ is well-defined, that is, that $\tau_2{\sf g}$ is indeed a $\kappa$-word for every ${\sf g}\in\Gamma$, follows from
 the fact that $\eta'_k{\sf g}$ is a $\kappa$-word and $\theta_k$ transforms $\kappa$-words into $\kappa$-words
 (see~\cite{Costa&Nogueira&Teixeira:2015b}). Next, for each vertex ${\sf v}$, consider the length $k$ words
 ${\tt i}_{\sf v}={\tt i}_k\eta'_k{\sf v}={\tt i}_k\eta{\sf v}$ and  ${\tt t}_{\sf v}={\tt t}_k\eta'_k{\sf v}={\tt t}_k\eta{\sf v}$. We let
 $$\tau_{1}{\sf v}=\lambda_k  {\tt i}_{\sf v}\quad{and}\quad \tau_{3}{\sf v}=\varrho_k  {\tt t}_{\sf v},$$
  where the mappings $\lambda_k$ and $\varrho_k$ were defined in~\eqref{eq:lambda_varrho}. Note that,  by~\eqref{eq:eta'k_v},  ${\tt t}_{\sf v}=x_{\sf v}y_{\sf v} z_{\sf v}$. Moreover, the occurrence  of $y_{\sf v}$ shown in this factorization is the last occurrence of a border in ${\tt t}_{\sf v}$. Hence, the right-hand ${\tt r}_{{\tt t}_{\sf v}}$ of ${\tt t}_{\sf v}$ is precisely $z_{\sf v}$. Therefore, one has $$\tau_{1}{\sf
v}=\lambda_k{\tt i}_{\sf v}={\tt l}'_{{\tt i}_{\sf v}}{\tt e}_{{\tt i}_{\sf v}}{\tt e}_{{\tt i}_{\sf v}}^\omega\quad\mbox{and}\quad\tau_{3}{\sf v}=\varrho_k {\tt t}_{\sf v}={\tt e}_{{\tt t}_{\sf v}}^\omega {\tt l}''_{{\tt t}_{\sf v}}z_{\sf v}.$$

  Consider now an arbitrary edge ${\sf e}$.  Suppose  that $\eta{\sf e}$  is a finite word. Then, $\eta{\sf e}$ is a letter $a_{\sf e}$
  and $\eta'_k{\sf e}$ is also $a_{\sf e}$ in this case. Then  $\tau_2{\sf e}=\theta_k a_{\sf e}=1$ because $\theta_k$
  is a $k$-superposition homomorphism. Since we want $\eta'{\sf e}$ to be $a_{\sf e}$, we then define, for instance,
  $$\tau_1{\sf e}=a_{\sf e}\quad\mbox{and}\quad\tau_3{\sf e}=1.$$
  Suppose at last that $\eta{\sf e}$ (and so also $\eta'_k{\sf e}$) is an infinite pseudoword. We let
  $$\tau_3{\sf e}=\varrho_k  {\tt t}_{\sf e}$$
   and notice that $\tau_3{\sf e}=\tau_3\omega{\sf e}$. Indeed, as $\eta'_k$ is a $\V*\D_k$-solution of $\Sigma_\Gamma$, it follows from~\eqref{eq:char-psid-VDk} that ${\tt t}_{\sf e}={\mathtt
t}_k\eta'_k{\sf e}={\mathtt t}_k\eta'_k\omega{\sf e}={\tt t}_{\omega{\sf e}}$. The definition of $\tau_{1}{\sf e}$ is more elaborate.
   Let ${\sf v}$ be the vertex $\alpha{\sf e}$ and consider the word ${\tt t}_{\sf v}{\tt i}_{\sf e}=a_1\cdots a_{2k}$. This word
has $r=k+1$ factors of length $k$. Suppose that $\theta_k({\tt t}_{\sf v}{\tt i}_{\sf e})$ is
$e_1^{\omega} f_1 e_2^{\omega} f_2\cdots   e_{r-1}^{\omega} f_{r-1} e_{r}^{\omega}$ and consider its reduced form
$$\theta_k({\tt t}_{\sf v}{\tt i}_{\sf e})=e_{1}^{\omega} {\bar f}_1 {e}_{n_2}^\omega {\bar f}_2\cdots   {e}_{n_q}^{\omega} {\bar f}_{q}.$$
Notice that ${\tt t}_{\sf v}{\tt i}_{\sf e}={\bar f}_{0}{\bar f}_1\cdots {\bar f}_{q}{\bar f}_{q+1}$  for some words ${\bar f}_{0},{\bar f}_{q+1}\in A^*$.
 Hence, there is a (unique) index
$m\in\{1,\ldots,q\}$  such that ${\tt t}_{\sf v}={\bar f}_0{\bar f}_1\cdots {\bar f}_{m-1}{\bar f'}_{m}$ and
${\bar f}_{m}={\bar f'}_{m}{\bar f''}_{m}$ with ${\bar f'}_{m}\in A^*$ and ${\bar f''}_{m}\in A^+$. Then
$\theta_k({\tt t}_{\sf v}{\tt i}_{\sf e})=\beta_1\beta_2$, where $\beta_1=e_{1}^\omega {\bar
f}_1{e}_{n_2}^\omega {\bar f}_2\cdots e_{n_m}^\omega {\bar f'}_{m}$ and $\beta_2={\bar
f''}_{m}e_{n_{m+1}}^\omega {\bar f}_{m+1}\cdots e_{n_q}^\omega {\bar f}_q$ and we let
$$\tau_1{\sf e}=\beta_2={\bar f''}_{m}e_{n_{m+1}}^\omega {\bar f}_{m+1}\cdots e_{n_q}^\omega {\bar f}_q.$$
Note that the word $\beta'_2={\bar f''}_{m}{\bar f}_{m+1}\cdots {\bar f}_q$ is $a_{k+1}\cdots a_{j_{r}}$, whence
$\beta'_2e_{r}^{\omega}=\lambda_k {\tt i}_{\sf e}$.

 The next lemma is a key result that justifies the definition of the $\widehat{\ \ }$-operation.
\begin{lemma}\label{lemma:avanco_edge}
Let ${\sf e}$ be an edge such that $\eta{\sf e}$ is infinite. Then, with the above notation, $\beta_1=\tau_{3}{\sf v}$ and  so  $\theta_k({\tt t}_{\sf v}{\tt i}_{\sf
e})=(\tau_{3}{\sf v})(\tau_{1}{\sf e})$. Moreover, $\delta\tau_1{\sf e}=\delta\lambda_k {\tt i}_{\sf e}$.
\end{lemma}
\begin{proof} We begin by recalling that ${\tt t}_{\sf v}{\tt i}_{\sf e}=a_1\cdots a_{2k}$ and
$$\theta_k({\tt t}_{\sf v}{\tt i}_{\sf e})=e_1^{\omega} f_1 e_2^{\omega} f_2\cdots   e_{r-1}^{\omega} f_{r-1} e_{r}^{\omega}=e_{1}^{\omega} {\bar f}_1 {e}_{n_2}^\omega {\bar f}_2\cdots   {e}_{n_q}^{\omega} {\bar f}_{q},$$
 where $e_p$ is the essential factor ${\tt e}_{w_p}=a_{i_p}\cdots  a_{j_p}$ of the word $w_p=a_p \cdots a_{k+p-1}$  and
$f_p=a_{j_p+1} \cdots  a_{j_{p+1}}$ for each $p$. Note also that $\lambda_k {\tt i}_{\sf
e}=\beta'_2e_{r}^{\omega}$, $e_{r}$ is a suffix of $\beta'_2$ and $\delta e_r$ is idempotent. So, to prove the
equality $\delta\tau_1{\sf e}=\delta\lambda_k {\tt i}_{\sf e}$ it suffices to show that $\delta\tau_1{\sf
e}=\delta\beta'_2$. We know from~\eqref{eq:eta'k_v} that  ${\tt t}_{\sf v}=x_{\sf v}y_{\sf v} z_{\sf v}$ with $1\leq|x_{\sf v}|\leq Q$. So, $x_{\sf v}=a_1\cdots a_{h-1}$, $y_{\sf v}=a_{h}\cdots a_{M+h-1}$ and $z_{\sf v}=a_{M+h}\cdots a_{k}$ for some $h\in\{2,\ldots,Q+1\}$. There are two cases to verify.

\begin{enumerate}[label=\textbf{Case \arabic*.},ref=\arabic*,labelwidth=-10ex,labelsep=*,leftmargin=3ex,itemindent=7ex,parsep=1ex]
\item  $y_{\sf v}$ is a non-periodic border. Consider the factor $w_h=a_h \cdots a_{k+h-1}$ of
${\tt t}_{\sf v}{\tt i}_{\sf e}$. By the choice of $M$ and $k$, the prefix $y_{\sf v}$ is the only occurrence of
a border in $w_h$. Hence, $M$ is the last bound of $w_h$ and, so, its splitting point. It follows that
$w_h=y_{\sf v}\cdot z_{\sf v}a_{k+1}\cdots a_{k+h-1}$ is the splitting factorization of $w_h$.  Therefore, as
one can  verify for an arbitrary $p\in\{1,\ldots,h\}$, there is only one occurrence of a border in $w_{p}$, precisely $y_{\sf v}$, and the splitting factorization of $w_{p}$ is
  $$w_{p}=a_{p}\cdots a_{h-1}y_{\sf v}\cdot z_{\sf v}a_{k+1}\cdots a_{k+p-1},$$
whence $e_p=e_1$ with $j_p=j_1\leq M+h-1$ and, so, $f_p=1$ for $p< h$. So, the prefix $e_1^{\omega} f_1
e_2^{\omega} \cdots f_{h-1} e_{h}^{\omega}$ of $\theta_k({\tt t}_{\sf v}{\tt i}_{\sf e})$ reduces to
$e_1^{\omega}$. Consider now the factor $w_{h+1}=a_{h+1} \cdots a_{k+h}$. Hence, either $w_{h+1}$ does not have
a last bound or $k$ is its last bound. In both situations, the splitting point of $w_{h+1}$ is $k$ and its
splitting factorization is $w_{h+1}=w_{h+1}\cdot 1$. Therefore, one deduces from Lemma~\ref{lemma:correct_order}
that, for every $p\in\{h+1,\ldots,r\}$, the occurrence
   $a_{i_p}\cdots  a_{j_p}$  of the essential factor ${\tt e}_{w_p}$  in $w_p$ is, in fact, an occurrence
   in the suffix $w'=a_{k+h-E}\cdots a_{2k}=a_{M+L+h}\cdots a_{2k}$ of ${\tt t}_{\sf v}{\tt i}_{\sf e}$.
   Since $|x_{\sf v}y_{\sf v}|=M+h-1$ and $|z_{\sf v}|\leq L$, it follows that $k=|x_{\sf v}y_{\sf v} z_{\sf v}|<M+L+h$,
   whence $w'$ is a suffix of ${\tt i}_{\sf e}$ and so $k<i_p<j_p$ for all $p\in\{h+1,\ldots,r\}$.
    This means, in particular, that the $\omega$-power $e_{h+1}^{\omega}$ is introduced at the
    suffix ${\tt i}_{\sf e}$ of ${\tt t}_{\sf v}{\tt i}_{\sf e}$. Hence
    $\beta_1=e_1^{\omega} f_1 e_2^{\omega} \cdots f_{h-1} e_{h}^{\omega}a_{j_{h}+1}\cdots a_k$
    and its reduced form is  $e_1^{\omega}a_{j_{1}+1}\cdots a_k=\tau_{3}{\sf v}$, which proves that
    $\beta_1$ and $\tau_{3}{\sf v}$ are the same $\kappa$-word. Moreover, from
$k<i_{p}$, one deduces that the word $e_p$ is a suffix of $a_{k+1}\cdots a_{j_p}$, which proves that
$\delta\tau_1{\sf e}=\delta\beta'_2$.

\item  $y_{\sf v}$ is a periodic border. Let $u$ be the root of $y_{\sf v}$.  Then, since $M$ was fixed as a
multiple of $|u|$, $y_{\sf v}=u^{M_u}$ where $M_u=\frac{M}{|u|}$. If the prefix $y_{\sf v}$
   is the only occurrence of a border in $w_h$, then one deduces the lemma as in Case 1 above. So, we
   assume that there is another occurrence of a border $y$ in $w_h$. Hence, by Proposition~\ref{prop:borders} and
    the choice of $M$ and $k$, $y$ is precisely $y_{\sf v}$. Furthermore, since $u$ is a Lyndon word and $k=M+Q$
    with $Q<M$, $w_h=y_{\sf v}u^{d}w'_h$ for some positive integer $d$ and some word $w'_h\in A^*$ such that $u$ is not a prefix of
    $w'_h$. Notice that, since $u$ is not a prefix of $z_{\sf v}$ by definition of this word, $z_{\sf v}$ is a proper prefix of $u$.
     On the other hand $w_h=u^dy_{\sf v}w'_h$ and the occurrence of $y_{\sf v}$ shown in this factorization is the last occurrence of $y_{\sf v}$ in $w_h$. Thus,
     $$w_h=u^dy_{\sf v}\cdot w'_h$$
     is the splitting factorization of $w_h$. Therefore $\widehat{w_h}=u^dy_{\sf v}(u^{n_{\!S}})^\omega w'_h$ and
     $e_h=u^{n_{\!S}}$. More generally,  for any $p\in\{1,\ldots,h\}$, $y_{\sf v}$ is a factor of $w_p$ and it
   is the only border that occurs in $w_p$. Hence, the splitting point of $w_p$ is periodic and
   $e_p=u^{n_{\!S}}$. Moreover, as one can verify, $j_1=M+h-1$ and the prefix $e_1^{\omega} f_1
e_2^{\omega} \cdots f_{h-1} e_{h}^{\omega}$ of $\theta_k({\tt t}_{\sf v}{\tt i}_{\sf e})$ is $e_1^{\omega} (u
(e_1^{\omega})^{|u|})^d$ and so, analogously to Case 1, it reduces to $e_1^{\omega}u^d$. Since $z_{\sf v}$ is
a proper prefix of $u$ and $d\geq 1$, $k<j_h$. This allows already deduce that the reduced form of
    $\beta_1$ is  $(u^{n_{\!S}})^\omega z_{\sf v}=\tau_{3}{\sf v}$, thus concluding the proof of the first part of the lemma. Now, there
    are two possible events. Either $m=q$ and $\beta_2={\bar f''}_{m}=\beta'_2$, in which case
     $\delta\tau_1{\sf e}=\delta\beta'_2$ is trivially verified. Or $m\neq q$ and the $\omega$-power
     $e_{n_{m+1}}^\omega$ was not eliminated in the reduction process of $\theta_k({\tt t}_{\sf v}{\tt i}_{\sf
     e})$. This means that the splitting point of the word $w_{n_{m+1}}$ is not determined by one of
     the occurrences of the border $y_{\sf v}$ in the prefix $a_1\cdots a_{k+h-1}$ of ${\tt t}_{\sf v}{\tt i}_{\sf
     e}$. Then, as in Case 1 above, one deduces that
    $k<i_{p}$ for each $p\in\{ n_{m+1},\ldots, r\}$ and, so, that $\delta\tau_1{\sf e}=\delta\beta'_2$.
  \end{enumerate}
  In both cases  $\beta_1=\tau_{3}{\sf v}$ and $\delta\tau_1{\sf e}=\delta\lambda_k {\tt i}_{\sf e}$. Hence, the proof of the lemma is complete.
\end{proof}

Notice that, as shown in the proof of Lemma~\ref{lemma:avanco_edge} above, if a vertex ${\sf v}$ is such that
$y_{\sf v}$ is a periodic border with root $u$, then $\tau_3{\sf v}=(u^{n_{\!S}})^\omega z_{\sf v}$. So, the  definition
of the mapping $\tau_3$ on vertices assures condition $\mathcal C_2(\Gamma,\eta,\eta')$.
%%%%%%%%%%%%%%%%%%%%%%%%%%%%%%%%%%%%%%%%%%%%%%%%%%%%%%%%%%%%%%%%%%%%%%%%%%%%%%%%%%%%%%%%%%%%%%%%%%%%%%%%%%%%%%%%%%%%%%%%%%%%%%%%
%%%%%%%%%%%%%%%%%%%%%%%%%%%%%%%%%%%%%%%%%%%%%%%%%%%%%%%%%%%%%%%%%%%%%%%%%%%%%%%%%%%%%%%%%%%%%%%%%%%%%%%%%%%%%%%%%%%%%%%%%%%%%%%%
\subsection{Proof that $\eta'$ is a  $(\V*\D,\kappa)$-solution}
This section will be dedicated to showing that $\eta'$ is a $(\V*\D,\kappa)$-solution of $\Sigma_\Gamma$ with respect to the pair $(\varphi,\delta)$ verifying conditions $\mathcal C_1(\Gamma,\eta,\eta')$ and $\mathcal C_3(\Gamma,\eta,\eta')$.

We begin by noticing that $\eta'{\sf g}$ is a $\kappa$-word for every ${\sf g}\in\Gamma$. Indeed, as observed above, each $\tau_2{\sf g}$ is a $\kappa$-word. That both $\tau_1{\sf g}$ and $\tau_3{\sf g}$ are $\kappa$-words too, is easily seen by their definitions.  Let us now show  the following properties.
 \begin{proposition}\label{prop:same-value-S}
 Conditions $\delta\eta'=\varphi$, $\mathcal C_1(\Gamma,\eta,\eta')$ and $\mathcal C_3(\Gamma,\eta,\eta')$ hold.
 \end{proposition}
\begin{proof}
As $\eta'_k$ is a $\V*\D_k$-solution of $\Sigma_\Gamma$ with respect to $(\varphi,\delta)$ and, so, the equality
$\delta\eta'_k=\varphi$ holds, to deduce that $\delta\eta'=\varphi$ holds it suffices to establish the equality
$\delta  \eta' =\delta  \eta'_k$. Consider first a vertex ${\sf v}\in \Gamma$. Then $\tau_{1}{\sf
v}=\lambda_k{\tt i}_{\sf v}={\tt l}'_{{\tt i}_{\sf v}}{\tt e}_{{\tt i}_{\sf v}}{\tt e}_{{\tt i}_{\sf v}}^\omega$
and  $\tau_{3}{\sf v}=\varrho_k {\tt t}_{\sf v}={\tt e}_{{\tt t}_{\sf v}}^\omega {\tt l}''_{{\tt t}_{\sf v}}z_{\sf v}$. In this case, the equality $\delta \eta'_k{\sf v}=\delta \eta'{\sf v}$ is a direct
application of~\cite[Proposition 5.3]{Costa&Nogueira&Teixeira:2015b}, where the authors proved that
\begin{equation}\label{eq:preservation_delta}
\delta \pi=\delta\bigl((\lambda_k  {\tt i}_k\pi)(\theta_k \pi)(\varrho_k  {\tt t}_k\pi)\bigr)
\end{equation}
 for every pseudoword $\pi$. Moreover, by definition of the $\widehat{\ \ }$-operation, $|{\tt l}'_{{\tt i}_{\sf
v}}|>L$. Therefore, $\eta{\sf v}$ and $\eta'{\sf v}$ are of the form $\eta{\sf v}=u\pi$ and $\eta'{\sf v}=u\pi'$
with $u\in A^L$ and $\delta\pi=\delta\pi'$. So, condition $\mathcal C_3(\Gamma,\eta,\eta')$ holds.

Consider next an edge ${\sf e}\in \Gamma$. If $\eta'_k{\sf e}$ is a finite word $a_{\sf e}$, then $\eta' {\sf
e}=(\tau_1{\sf e})(\tau_2{\sf e})(\tau_3{\sf e})=a_{\sf e}\cdot 1\cdot 1=a_{\sf e}=\eta'_k{\sf e}$, whence
$\delta \eta' {\sf e}=\delta  \eta'_k {\sf e}$ holds trivially. Moreover, since $\eta'_k {\sf e}=\eta {\sf e}$
in this case and every vertex is labeled under $\eta$ by an infinite pseudoword, it follows that condition
$\mathcal C_1(\Gamma,\eta,\eta')$ holds. Suppose at last that $\eta'_k{\sf e}$ is infinite and let ${\sf
v}=\alpha{\sf e}$. Then $\tau_{3}{\sf e}=\varrho_k {\tt t}_{\sf e}$. On the other hand, by
Lemma~\ref{lemma:avanco_edge}, $\delta\tau_1{\sf e}=\delta\lambda_k {\tt i}_{\sf e}$. Hence,
by~\eqref{eq:preservation_delta} and since $\delta$ is a homomorphism, $\delta \eta'{\sf e}=\delta
\bigl((\tau_1{\sf e})  (\tau_2{\sf e})  (\tau_3{\sf e})\bigr)=\delta\bigl((\lambda_k  {\tt i}_{\sf
e})(\theta_k\eta'_k{\sf e})(\varrho_k  {\tt t}_{\sf e})\bigr)=\delta \eta'_k{\sf e}$. This ends the proof of the
proposition.
\end{proof}

Consider an arbitrary edge ${\sf v}\xrightarrow {\sf e}{\sf w}$ of $\Gamma$. To achieve the objectives of this section it remains to prove that  $\V*\D$ satisfies $(\eta' {\sf v})(\eta'{\sf e})= \eta' {\sf w}$. Since $\eta'_k$ is a
$\V*\D_k$-solution of $\Sigma_\Gamma$, $\V*\D_k$ satisfies $(\eta'_k {\sf v})(\eta'_k{\sf e})= \eta'_k {\sf w}$. Hence, by~\eqref{eq:char-psid-VDk},
${\tt i}_{\sf v}={\tt i}_k\bigl((\eta'_k {\sf v})(\eta'_k{\sf e})\bigr)={\tt i}_k(\eta'_k{\sf w})={\tt i}_{\sf w}$ and ${\tt t}_k\bigl((\eta'_k {\sf v})(\eta'_k{\sf e})\bigr)={\tt t}_k(\eta'_k{\sf w})={\tt t}_{\sf w}$. Thus, $\tau_{1}{\sf v}=\lambda_k{\tt i}_{\sf v}={\tt l}'_{{\tt i}_{\sf v}}{\tt e}_{{\tt i}_{\sf v}}{\tt e}_{{\tt i}_{\sf v}}^\omega={\tt l}'_{{\tt i}_{\sf w}}{\tt e}_{{\tt i}_{\sf w}}{\tt e}_{{\tt i}_{\sf w}}^\omega=\lambda_k{\tt i}_{\sf w}=\tau_{1}{\sf w}$ and  $\tau_{3}{\sf w}=\varrho_k {\tt t}_{\sf w}={\tt
e}_{{\tt t}_{\sf w}}^\omega {\tt l}''_{{\tt t}_{\sf w}}z_{\sf w}$. As shown in the proof of~\cite[Proposition 5.4]{Costa&Nogueira&Teixeira:2015b}, it then follows that $\V*\D$ satisfies ${\tt e}_{{\tt i}_{\sf w}}^\omega\theta_k\bigl((\eta'_k {\sf v})(\eta'_k {\sf e})\bigr){\tt e}_{{\tt t}_{\sf w}}^\omega= {\tt e}_{{\tt i}_{\sf w}}^\omega\theta_k(\eta'_k {\sf w}){\tt e}_{{\tt t}_{\sf w}}^\omega$ and, so,
\begin{equation}\label{eq:VD_models}
\V*\D\models (\tau_{1}{\sf v})\theta_k\bigl((\eta'_k {\sf v})(\eta'_k {\sf e})\bigr)(\tau_{3}{\sf w})=(\tau_{1}{\sf w})\theta_k(\eta'_k {\sf w})(\tau_{3}{\sf w})=\eta' {\sf w}.
\end{equation}
On the other hand, from the fact that $\theta_k$ is a $k$-superposition homomorphism one deduces
\begin{equation}\label{eq:theta_product}
\theta_k\bigl((\eta'_k {\sf v})(\eta'_k {\sf e})\bigr)=\theta_k(\eta'_k {\sf v})\theta_k\bigl({\tt t}_{\sf v}(\eta'_k {\sf
e})\bigr)=\theta_k(\eta'_k {\sf v})\theta_k({\tt t}_{\sf v}{\tt i}_{\sf e})\theta_k(\eta'_k {\sf e}).
\end{equation}

Suppose that $\eta'_k{\sf e}$ is an infinite pseudoword. In this case ${\tt t}_{\sf e}={\tt t}_{\sf w}$, whence
$\tau_{3}{\sf e}=\tau_{3}{\sf w}$. Moreover, by Lemma~\ref{lemma:avanco_edge}, $\theta_k({\tt t}_{\sf v}{\tt
i}_{\sf e})=(\tau_{3}{\sf v})(\tau_{1}{\sf e})$. Therefore, by conditions~\eqref{eq:VD_models}
and~\eqref{eq:theta_product}, $\V*\D$ satisfies $(\eta' {\sf v})(\eta'{\sf e})=\eta' {\sf w}$. Assume now that
$\eta'_k{\sf e}$ is a finite word, whence $\eta'_k{\sf e}=a_{\sf e}\in A$ and $\eta'{\sf e}=a_{\sf e}$. Since
$\eta$ is a $\D$-solution of $\Sigma_\Gamma$, $\D\models(\eta'{\sf v})a_{\sf e}=\eta'{\sf w}$ and, thus, ${\bf
d}_{\sf v}a_{\sf e}={\bf d}_{\sf w}$. Hence the left-infinite words ${\bf d}_{\sf v}$ and ${\bf d}_{\sf w}$ are
confinal and, so, $\propto$-equivalent. Hence ${\bf d}_{\sf v}=y_\Delta z_{\sf v}$, ${\bf d}_{\sf w}=y_\Delta
z_{\sf w}$ and $y_{\sf v}=y_{\sf w}={\tt t}_k y_\Delta$, where $\Delta$ is the $\propto$-class of  ${\bf d}_{\sf
v}$ and ${\bf d}_{\sf w}$. It follows that $y_\Delta z_{\sf v}a_{\sf e}=y_\Delta z_{\sf w}$ and  ${\tt
t}_k\bigl({\tt t}_{\sf v}a_{\sf e})={\tt t}_{\sf w}$. In this case,  $\theta_k\bigl((\eta'_k {\sf v})(\eta'_k {\sf e})\bigr)=
 \theta_k(\eta'_k {\sf v})\theta_k({\tt t}_{\sf v}a_{\sf e})$. On the other hand, ${\tt t}_{\sf v}a_{\sf e}=a_1\cdots a_ka_{k+1}=a_1{\tt t}_{\sf w}$ is a word of length $k+1$ and, so, $\theta_k({\tt t}_{\sf v}a_{\sf e})=\psi_k({{\tt t}_{\sf v}a_{\sf e}})$ is of the form
 $$\theta_k({\tt t}_{\sf v}a_{\sf e})=e_1^{\omega} fe_2^{\omega}.$$
 The splitting factorizations of ${\tt t}_{\sf v}$ and ${\tt t}_{\sf w}$ are, respectively, ${\tt t}_{\sf v}=x_{\sf v}y_{\sf v}\cdot z_{\sf v}$ and ${\tt t}_{\sf w}=x_{\sf w}y_{\sf w}\cdot z_{\sf w}$. Since $y_{\sf v}=y_{\sf w}$, it follows that $e_1={\tt e}_{{\tt t}_{\sf v}}={\tt e}_{{\tt t}_{\sf w}}=e_2$.

 Suppose that $z_{\sf v}a_{\sf e}=z_{\sf w}$. In this case it is clear that $f=1$, so that $\theta_k({\tt t}_{\sf v}a_{\sf e})={\tt e}_{{\tt t}_{\sf v}}^\omega$. Since $\theta_k(\eta'_k {\sf v})$ ends with ${\tt e}_{{\tt t}_{\sf v}}^\omega$, it then follows that $\theta_k\bigl((\eta'_k {\sf v})(\eta'_k {\sf e})\bigr)=\theta_k\eta'_k {\sf v}=\tau_{2}{\sf v}$.  Therefore, $(\tau_{1}{\sf v})\theta_k\bigl((\eta'_k {\sf v})(\eta'_k {\sf e})\bigr)(\tau_{3}{\sf w})=(\tau_{1}{\sf v})(\tau_{2}{\sf v})(\tau_{3}{\sf w})$. On the other hand,
 $$\tau_{3}{\sf w}=\varrho_k{\tt t}_{\sf w}={\tt e}_{{\tt t}_{\sf w}}^\omega {\tt l}''_{{\tt t}_{\sf w}}z_{\sf w}={\tt e}_{{\tt t}_{\sf v}}^\omega
 {\tt l}''_{{\tt t}_{\sf v}}z_{\sf v}a_{\sf e}=(\tau_{3}{\sf v})a_{\sf e}.$$
 So, by~\eqref{eq:VD_models}, one has that $\V*\D$ satisfies $(\eta' {\sf v})a_{\sf e}=(\tau_{1}{\sf v})(\tau_{2}{\sf v})(\tau_{3}{\sf v})a_{\sf e}=(\tau_{1}{\sf v})(\tau_{2}{\sf v})(\tau_{3}{\sf w})=\eta'{\sf w}$.

 Suppose now that $z_{\sf v}a_{\sf e}\neq z_{\sf w}$. In this case, one deduces from the equality $y_\Delta z_{\sf v}a_{\sf e}=y_\Delta z_{\sf w}$, that
 $y_\Delta$ is a periodic left-infinite word. Let $u$ be its root, so that $y_\Delta=\infe{u}$, ${\tt e}_{{\tt t}_{\sf v}}=u^{n_{\!S}}$ and ${\tt l}''_{{\tt t}_{\sf v}}={\tt l}''_{{\tt t}_{\sf w}}=1$. Since, by definition, $u$ is a primitive word which is not a prefix of $z_{\sf v}$ nor a prefix of $z_{\sf w}$, we conclude that $z_{\sf v}a_{\sf e}=u$ and $z_{\sf w}=1$. In this case $f=u$,  whence $\theta_k({\tt t}_{\sf v}a_{\sf e})={\tt e}_{{\tt t}_{\sf v}}^\omega u$. Then, $\theta_k\bigl((\eta'_k {\sf v})(\eta'_k {\sf e})\bigr)=(\theta_k\eta'_k {\sf v})u=(\tau_{2}{\sf v})u$.  Therefore, $(\tau_{1}{\sf v})\theta_k\bigl((\eta'_k {\sf v})(\eta'_k {\sf e})\bigr)(\tau_{3}{\sf w})=(\tau_{1}{\sf v})(\tau_{2}{\sf v})u(\tau_{3}{\sf w})$. Moreover,
 $$u(\tau_{3}{\sf w})=u{\tt e}_{{\tt t}_{\sf w}}^\omega {\tt l}''_{{\tt t}_{\sf w}}z_{\sf w}=u(u^{n_{\!S}})^\omega=(u^{n_{\!S}})^\omega u={\tt e}_{{\tt t}_{\sf v}}^\omega {\tt l}''_{{\tt t}_{\sf v}} z_{\sf v}a_{\sf e}=(\tau_{3}{\sf v})a_{\sf e}.$$
 Therefore, using~\eqref{eq:VD_models}, one deduces as above that $\V*\D$ satisfies $(\eta' {\sf v})a_{\sf e}=\eta'{\sf w}$.

We have proved the main theorem of the paper.
\begin{theorem}\label{theo:main}
 If \V\  is  $\kappa$-reducible, then $\V*\D$ is  $\kappa$-reducible.
 \end{theorem}

 This result applies, for instance, to the pseudovarieties $\Sl$, $\G$, $\J$ and ${\bf R}$. Since the
 $\kappa$-word problem for the pseudovariety ${\bf LG}$ of local groups is already solved~\cite{Costa&Nogueira&Teixeira:2015},
 we obtain the following corollary.
 \begin{corollary}
 The pseudovariety $\LG$ is  $\kappa$-tame.
 \end{corollary}

%%%%%%%%%%%%%%%%%%%%%%%%%%%%%%%%%%%%%%%%%%%%%%%%%%%%%%%%%%%%%%%%%%%%%%%%%%%%%%%%%%%%%%%%%%%%%%%%%%%%%%%%%
\paragraph{\textbf{Final remarks.}} In this paper we fixed our attention on the canonical signature $\kappa$, while in~\cite{Costa&Nogueira&Teixeira:2015b} we dealt with a more generic class of signatures $\sigma$ verifying certain undemanding conditions. Theorem~\ref{theo:main} is still valid for such generic signatures $\sigma$ but we preferred to treat only the instance of the signature $\kappa$ to keep the proofs clearer and a little less technical.

 %%%%%%%%%%%%%%%%%%%%%%%%%%%%%%%%%%%%%%%%%%%%%%%%%%%%%%%%%%%%%%%%%%%%%%%%%%%%%%%%%%%%%%%%%%%%%
%%%%%%%%%%%%%%%%%%%%%%%%%%%%%%%%%%%%%%%%%%%%%%%%%%%%%%%%%%%%%%%%%%%%%%%%%%%%%%%%%%%%%%%%%%%%

\end{document}